\documentclass[a4paper]{amsart} 
\usepackage{amssymb,times, amscd,amsmath,amsthm, xypic}
\usepackage[all]{xy}
\usepackage{geometry}
\usepackage[dvips]{graphicx, color}
\usepackage{mathrsfs}
\author{Shoji Yokura}

\address
{Graduate School of Science and Engineering,
Kagoshima University, 21-35 Korimoto 1-chome, Kagoshima 890-0065, Japan}
\email {yokura@sci.kagoshima-u.ac.jp}
\title 
{Hilali conjecture and complex algebraic varieties}

\numberwithin{equation}{section}
\newtheorem{thm}[equation]{Theorem}
\newtheorem{pro}[equation]{Proposition}

\newtheorem{qu}[equation]{Question}
\newtheorem{cor}[equation]{Corollary}
\newtheorem{con}[equation]{Conjecture}
\newtheorem{lem}[equation]{Lemma}
\theoremstyle{definition}
\newtheorem{ex}[equation]{Example}
\newtheorem{defn}[equation]{Definition}

\newtheorem{rem}[equation]{Remark}

\newcommand{\Q}{{\mathbb Q}}

\def\op{\operatorname}

\begin{document} 
 
\begin{abstract}
A simply connected topological space is called \emph{rationally elliptic} if the rank of its total homotopy group and its total (co)homology group are both finite. A well-known Hilali conjecture claims that for a rationally elliptic space its homotopy rank \emph{does not exceed} its (co)homology rank. In this paper, after recalling some well-known fundamental properties of a rationally elliptic space
 and giving some important examples of rationally elliptic spaces and rationally elliptic singular complex algebraic varieties for which the Hilali conjecture holds, we give some revised formulas and some conjectures. We also discuss some topics such as mixd Hodge polynomials defined via mixed Hodge structures on cohomology group and the dual of the homotopy group, related to the ``Hilali conjecture \emph{modulo product}", which is an inequality between the usual homological Poincar\'e polynomial and the homotopical Poincar\'e polynomial.
\end{abstract}

\maketitle

\section{Introduction}
Homotopy group and homology group are clearly two very fundamental and important invariants in geometry and topology; they are related to each other, just like ``two wheels of a car" or ``two sides of a coin". The Hilali conjecture\footnote{We note that there are also \emph{relative Hilali conjectures} for a map instead of a space (see \cite{CHHZ, YY1, YY2, Yo3, ZCH}).} is a very simple inequality concerning the dimensions (called ``homotopy" and ``homology" dimensions) of these two invariants:
\begin{equation}\label{hilali}
\op{dim} \left (\pi_*(X) \otimes \Q \right) \leq \op{dim}  H_*(X; \Q ) \quad ( < \infty)
\end{equation}
Usually people may think (as the author thought before) that for most spaces (excluding strange or very wild spaces) these dimensions are both finite, but it turns out that it is not the case, as remarked in \cite[\S 32 Elliptic spaces]{FHT}. Namely, in a sense, topological spaces whose homotopy and homology dimensions are both finite belongs to a special class of spaces, called \emph{rationally elliptic} spaces. If one of the dimensions is not finite, such a space is called \emph{rationally hyperbolic}. The Hilali conjecture claims the above inequality (\ref{hilali}) for \emph{any rationally elliptic space} $X$ belonging to this ``special" class, namely, in words
$$ \text{\emph{homotopy dimension never exceeds homology dimension.}}$$
Complex algebraic varieties are certainly ``special" spaces in the category of topological spaces. In this paper we consider the Hilali conjecture for complex algebraic varieties. Unfortunately, even complex algebraic varieties are not ``special enough" to discuss the Hilali conjecture, because a complex algebraic variety $X$ has to be \emph{rationally elliptic} in order to see whether the Hilali conjecture holds or not. It turns out that, as proved by Stephen Halperin for the first time in 1977 and later by John B. Friedlander and S. Halperin in 1979, rationally elliptic spaces satisfy \emph{``very stringent restrictions or properties} (see \cite[\S 32 Elliptic spaces]{FHT}). This fact might be good since one could restrict oneself to very special spaces, but bad since it would be hard to control such stringent properties.

The purposes of the present paper are as follows:
\begin{itemize}
\item To propagandize or advertise the Hilali conjecture,
\item To inform that for ``most spaces", or ``generically", either one of the homology dimension and the homotopy dimension is infinite, thus these two dimensions are both finite only for a very special class of spaces and such a space is called \emph{rationally elliptic} and the Hilali conjecture is for such a space.
\item To show that spheres and complex projective spaces are fundamental and important spaces for discussing the Hilali conjecture and the notion of rational ellipticity.
\item To give revised versions of some known formulas and theorems.
\item To give some reasonable questions and conjectures in order to deal with the Hilali conjecture furthermore.
\end{itemize}
The paper is organized as follows. In \S 2 we recall the notion of rational ellipticity, in particular its importance in connection with well-known conjectures such as Bott conjecture, Hopf conjecture and Gromov conjecture in Riemannian geometry.
In \S 3 we recall the Hilali conjecture and see that spheres and complex projective spaces (which are very fundamental and important spaces, building other interesting spaces and complex algebraic varieties) and Cartesian products of them are rationally elliptic and also satisfy the Hilali conjecture. Related to these examples we discuss singular complex algebraic varieties homeomorphic and/or homotopic to the complex projective spaces. In \S 4 we discuss stringent properties of a rationally elliptic space, which were discovered by J. B. Friedlander and S. Halperin, in particular as typical models we consider spheres, complex projective spaces and Cartesian products of them. We also give some revised or strengthened versions of some known formulas and theorems. In \S\S 5-7 we recall some known results about rationally elliptic toric and K\"ahler manifolds and based on these results and discussion in this paper we give a conjecture claiming that \emph{the Hilali conjecture would hold for any complex algebraic variety, singular or non-singular, provided that it is rationally elliptic.} In \S\S 8-10 we discuss what is called ``Hilali conjecture \emph{modulo product}" concerning the usual (homological) Poincar\'e polynomial and the homotopical Poincar\'e polynomial and also the mixed Hodge polynomial and the homotopical mixed Hodge polynomial (based on the existence of mixed Hodge structure), which specialize to the corresponding Poincar\'e polynomials. At the very end we give a conjecture (due to Anatoly Libgober) claiming that \emph{if a quasi-projective variety is rationally elliptic, then the mixed Hodge structure is of Hodge--Tate type}.

Finally we remark that it is very standard or usual to discuss rationally elliptic spaces and the Hilali conjecture, in rational homotopy theory, appealing to Sullivan's model theory, but in this paper we do not do so for this kind of presentation.

\section{Rationally elliptic space}\label{res}
We let 
$$\pi_*(X)\otimes \Q := \bigoplus_{k \geq 1} \pi_k(X) \otimes \Q \quad  \text{and} \quad  H_*(X; \Q ):= \bigoplus_{k \geq 0} H_k(X;\Q).$$
A simply connected 
topological space $X$ is called \emph{rationally elliptic}\footnote{This terminology has nothing to do with an \emph{elliptic} curve in Algebraic Geometry. In \cite{Hal} this name was not used. In \cite[\S 1 Introduction]{FH} it was called a \emph{space of type F} and in \cite{GH} this name seems to start being used.} 
if 
$$\dim \left (\pi_*(X) \otimes \Q \right) = \sum_{k \geq 2} \dim \left (\pi_k(X) \otimes \Q \right) < \infty, \, \, \dim H_*(X;\Q) = \sum_{k \geq 0} \dim  H_k(X; \Q ) < \infty.$$
Here we note that these dimensions are respectively the ranks of the total homotopy group $\pi_*(X) = \bigoplus_{k \geq 2} \pi_k(X)$ and the total homology group $H_*(X;\mathbb Z) = \bigoplus_{k \geq 0} H_k(X; \mathbb Z )$.
From now on, $\dim \left (\pi_*(X) \otimes \Q \right)$ and $\dim H_*(X;\Q)$ shall be simply called ``homotopy" and ``homology" dimension, respectively.
The cohomology $H^*(X;\Q)$ can be also used instead of the homology group, but by the universal coefficient theorem we have $\dim H^*(X;\Q)=\dim H_*(X;\Q)$.

Importance of this \emph{rational ellipticity}, in particular, comes from the following well-known conjecture attributed to Raoul Bott (cf. \cite{BB}) in Riemannian Geometry:
\begin{con}[Bott conjecture\footnote{This is supposed to be one of the central conjectures in Riemannian geometry}]
A compact simply connected Riemannian manifold with a non-negative sectional curvature is rationally elliptic.
\end{con}
As to this conjecture, we cite the following remark from Grove--Halperin's paper \cite[Introduction, p.380]{GH}:

\emph{``This conjecture has been attributed to Bott. Our interest in it was first stimulated by D. Toledo. An assertion equivalent to the conjecture is that the integers $\rho_p= \sum_{q\leq p}\dim H_q(\Omega M;\mathbb Q)$
 grow only sub-exponentially in $p$ (i.e., $\forall C>1, \forall k \in \mathbb N, \exists p>k : \rho_p <C^p$), in particular the principal result of \cite{BB} is a much weaker version of this conjecture."}

Hence Bott conjecture is sometimes called ``Bott--Grove--Halperin conjecture".
Later we will see that if $M$ is rationally elliptic and $\dim M=n$, then we have 
 \begin{enumerate}
\item  $\chi(M) \geq 0$.
 \item $\dim H_*(M;\mathbb Q) \leq 2^n$  
 \end{enumerate}
Thus, \emph{a positive answer to Bott conjecture} would imply the following Hopf--Chern conjecture and Gromov conjecture (in the case when $\mathbb F=\Q$), which are still open:
\begin{con}[Hopf conjecture \cite{Hopf}(also see \cite{BG}), 1953 or Chern conjecture \cite{Chern}, 1966]
A compact simply connected even-dimensional Riemannian manifold $M$ with non-negative sectional curvature has non-negative Euler characteristic, $\chi(M) \geq 0$.  
 \end{con}
 \begin{con}[Gromov conjecture \cite{Gromov},1981]
A simply connected complete manifold $M$ (of dimension $n$) with a non-negative sectional curvature satisfies $\dim H_*(M;\mathbb F) \leq 2^n$ for \emph{any field $\mathbb F$}. 
 \end{con} 
\begin{rem}
To be more precise about Gromov conjecture, Mikhael L. Gromov has proved that there is a constant $C(n)$ such that $\dim H_*(M;\mathbb F) \leq C(n)$ and then conjectured that the upper bound of $C(n)$ is $2^n$, i.e., $C(n) \leq 2^n$.
\end{rem}
As to Bott' conjecture, Xiaoyang Chen \cite{Chen} has recently proved 
\begin{thm} A simply connected Riemannian manifold with entire Grauert tube is rationally elliptic.
\end{thm}
L\'aszl\'o Lempert and R\'obert Sz\"oke \cite{LS} has proved that ``entire Grauert tube" implies ``non-negative sectional curvature", thus we can say that X. Chen has affirmatively solved Bott--Grove--Halperin conjecture \emph{under the stronger assumption} of ``entire Grauert tube".
 
 \section{Hilali conjecture}

Mohamed Rachid Hilali made the following conjecture in his thesis \cite{H} (also see a recent survey paper by Yves F\'elix and S. Halperin \cite[\S 10]{FH}): 
\begin{con}[Hilali Conjecture, 1990] If $X$ is a rationally elliptic space, then we have
$$\dim (\pi_*(X) \otimes \Q) \leq \dim H_*(X;\Q ).$$
\end{con}
The conjecture is very simple in the sense that it is just an inequality of the dimensions of homotopy and homology (which are fundamental, important, well-known and well-studied invariants in geometry and topology, even in mathematical physics), thus it seems to be quite tractable and to be solvable easily, but no one has found a counterexample yet, thus it is still open in even more than 30 years after it was conjectured.

\subsection{Fundamental examples for which the Hilali conjecture holds.}

\begin{ex}\label{sphere}
The following results follow from the well-known \emph{Serre Finiteness Theorem} \cite{Se}:

$$
\pi_i(S^{2k}) \otimes \Q 
=\begin{cases} 
 \Q & \, i=2k, 4k-1,\\
\, 0  & \,  i\not =2k, 4k-1,
\end{cases}  \qquad 
\pi_i(S^{2k+1}) \otimes \Q 
=\begin{cases} 
 \Q & \,  i=2k+1,\\
\, 0 & \, i \not =2k+1.
\end{cases} 
$$
Hence we see that
\begin{equation*}
2= \dim (\pi_*(S^{2k}) \otimes \Q) =\dim H_*(S^{2k};\Q )=1+ \dim H_{2k}(S^{2k};\Q )=2. 
\end{equation*}
\begin{equation*}
1= \dim (\pi_*(S^{2k+1}) \otimes \Q) < \dim H_*(S^{2k+1};\Q )=1+ \dim H_{2k+1}(S^{2k+1};\Q )=2. 
\end{equation*}
\end{ex}

\begin{ex}\label{proj}
$$\pi_k(\mathbb {CP}^n)\otimes \mathbb Q =
\begin{cases}
\mathbb Q \quad \text{for $k=2, 2n+1$,}\\
0 \quad \text{for $k\not = 2, 2n+1$,}\\
\end{cases}
$$
This follows from the long exact sequence of a fibration $S^1 \hookrightarrow S^{2n+1} \to \mathbb {CP}^n$:
$$\cdots \to \pi_k(S^1) \to \pi_k(S^{2n+1}) \to \pi_k(\mathbb {CP}^n) \to \pi_{k-1}(S^1) \to \pi_{k-1}(S^{2n+1}) \to \cdots $$
The following is known:
$$
H_k(\mathbb {CP}^n; \mathbb Q) =
\begin{cases}
\mathbb Q \quad \text{for $k=0$ and $2 \leq k \leq 2n$ for even $k$,}\\
0 \quad \text{otherwise.}\\
\end{cases}
$$
Hence we have 
$$2=\dim (\pi_*(\mathbb {CP}^n) \otimes \Q) \leq  \dim H_*(\mathbb {CP}^n;\Q )=1+n \, \, (n \geq 1).$$
\end{ex}

Before giving more examples, let us observe the following:
\begin{pro}\label{prop}
 If $X_1, \cdots, X_m$ are rationally elliptic and satisfy the Hilali conjecture, then $X_1 \times \cdots \times X_m$ is also rationally elliptic and satisfies the Hilali conjecture.
\end{pro}
\begin{proof}
Rational ellipticity of $X_1 \times \cdots \times X_m$ follows from homotopy being ``additive" and homology being ``multiplicative":
$$\dim (\pi_*(X_1 \times \cdots \times X_m) \otimes \Q) = \dim (\pi_*(X_1)\otimes \Q) + \cdots +  \dim (\pi_*(X_m)\otimes \Q) < \infty.$$
$$\dim H_*(X_1 \times \cdots \times X_m; \Q) = \dim H_*(X_1;\Q) \times  \cdots \times  \dim H_*(X_m;\Q) < \infty.$$
$X_1 \times \cdots \times X_m$ satisfying the Hilalic conjecture, i.e., 
$$\dim (\pi_*(X_1 \times \cdots \times X_m) \otimes \Q)  \leq \dim H_*(X_1 \times \cdots \times X_m; \Q),$$
which follows from Lemma \ref{lemma} (the proof of which is by induction, left for the reader) below.
\end{proof}
\begin{lem}\label{lemma}
Let $a_i, b_i$ ($1 \leq i \leq m$) be real numbers such that $a_i \leq b_i$ and $b_i \geq 1$. If either $2 \leq b_i$ or $0=a_i <b_i=1$ for each $i$, then $a_1 + \cdots +a_m \leq b_1 \times \cdots \times b_m$.
\end{lem}
$``0=a_i <b_i=1"$ comes from the following:
$$b_i=\dim H_*(X_i;\Q)=1 \Longrightarrow a_i=\dim (\pi_*(X_i)\otimes \Q)=0$$
which  follows from the following Serre Theorem (usually called Whitehead--Serre Theorem):
\begin{thm}\label{serre} If $f:X \to Y$ is a continuous map of simply connected spaces, the following are equivalent
\begin{enumerate}
\item $\pi_*(f)\otimes \Q:\pi_*(X)\otimes \Q \cong \pi_*(Y)\otimes \Q$ is an isomorphism.
\item $H_*(f;\Q):H_*(X;\Q) \cong H_*(Y;\Q)$ is an isomorphism.
\end{enumerate}
\end{thm}
Indeed, $\dim H_*(X;\Q)=1$, i.e., $H_0(X;\Q)=\Q$ and $H_i(X;\Q)=0 (i \geq 1)$, implies that $H_*(a_X;\Q):H_*(X;\Q)=\Q \cong H_*(pt;\Q)=\Q$ for a constant map $a_X:X \to pt$. 
Then the above Serre Theorem implies the homotopy isomorphism $\pi_*(a_X) \otimes \Q:\pi_*(X)\otimes \Q \cong \pi_*(pt)\otimes \Q=0$, i.e., $\dim (\pi_*(X_i)\otimes \Q)=0$.

\begin{ex} It follows from Proposition \ref{prop} above that 
$$
\text{$S^{n_1} \times \cdots \times S^{n_k}$, $\mathbb {CP}^{m_1} \times \cdots \times \mathbb {CP}^{m_s}$ and $S^{n_1} \times \cdots \times S^{n_k} \times \mathbb {CP}^{m_1} \times \cdots \times \mathbb {CP}^{m_s}$ }$$are rationally elliptic and satisfies the Hilali conjecture.
\end{ex}

\begin{ex}
A simply connected compact Lie group $G$ is rationally elliptic and satisfies the Hilali conjecture.
Because it is well-known (by Heinz Hopf) that $G$ is rationally homotopy equivalent to the product of odd-dimensional spheres, i.e., there is a map $f:G \to S^{2k_1+1} \times \cdots \times S^{2k_n +1}$ such that 
$\pi_*(f)\otimes \Q: \pi_*(G)\otimes \Q \cong \pi_*(S^{2k_1+1} \times \cdots \times S^{2k_n +1})\otimes \Q.$
Hence, by Theorem \ref{serre} (Serre theorem), $H_*(f;\Q): H_*(G;\Q) \cong H_*(S^{2k_1+1} \times \cdots \times S^{2k_n +1};\Q).$
\end{ex}

We note that $\mathbb {CP}^{m_1} \times \cdots \times \mathbb {CP}^{m_s}$ is a rationally elliptic \emph{K\"ahler manifold} and also a rationally elliptic \emph{smooth toric variety}. We will come back to rationally elliptic K\"ahler manifolds and smooth toric varieties later again.

{}
\subsection {Rationally elliptic singular varieties for which the Hilali conjecture holds.}
 $\mathbb {CP}^{m_1} \times \cdots \times \mathbb {CP}^{m_s}$ is a non-singular complex algebraic variety. How about a singular complex algebraic variety?
\begin{ex} A cuspidal curve is a singular curve whose singularities are cusps. A cuspidal curve is homeomorphic to $\mathbb {CP}^1$. Hence, it is a rationally elliptic and satisfies the Hilali conjecture.
\end{ex}
\begin{rem}
A classification of cuspidal curves seems to be still open. Mariusz Koras and Karol Palka \cite{KP} have recently proved that a cuspidal curve can have \emph{at most 4 cusps.}
\end{rem}

A cuspidal curve is a singular curve which is \emph{homeomorphic to $\mathbb {CP}^1$}. 
\begin{qu}\label{qu-p^n}
How about a singular complex algebraic variety\footnote{Note that this variety is not \emph{a fake projective space}, which is a non-singular complex algebraic variety which has the same Betti numbers as a complex projective space, but not isomorphic to it.} which is homeomorphic to $\mathbb {CP}^n \, \, (n \geq 2)$ or homotopic to $\mathbb {CP}^n \, \, (n \geq 1)$?
\end{qu}
As to $\mathbb {CP}^2$, Lawrence Brenton \cite{Br}\footnote{The author asked Alex Dimca for some work concerning Quesion \ref{qu-p^n} and then he immediately answered with his related works \cite{BD, CD}. In the reference of \cite{CD} is cited Brenton's paper \cite{Br}. The author would like to thank A. Dimca for sending these two papers \cite{BD, CD}.} has constructed singular surfaces which are homotopic to  $\mathbb {CP}^2$. As far as the author knows, there seems to be no such a paper available for $\mathbb {CP}^n$ with $n \geq 3$.

\begin{ex}
Let $C_1, \cdots, C_s$ are cuspidal curves and $B_1 , \cdots, B_t$ be Brenton's singular surfaces. Then $C_1 \times \cdots \times C_s$ and $B_1 \times  \cdots \times B_t$ are respectively $s$-dimensional and $2t$-dimensional singular varieties which are rationally elliptic and also satisfies the Hilali conjecture.  
$\mathbb {CP}^{m_1} \times \cdots \times \mathbb {CP}^{m_r} \times C_1 \times \cdots \times C_s \times B_1 \times \cdots \times B_t$ is a $2(m_1 + \cdots +m_r) +s +2t$-dimensional rationally elliptic singular variety and also satisfies the Hilali conjecture.
\end{ex}
These examples are trivial ones. Here is a naive question:

\begin{qu}\label{qu-char}
Is there a characterization of a rationally elliptic singular complex algebraic variety which satisfies the Hilali conjecture?
\end{qu}
As to the above Question \ref{qu-char}, we would like to pose the following very naive or simple-minded question:
\begin{qu} Is it correct that a rationally elliptic complex algebraic variety is always homeomorphic or homotopic to the product of complex projective spaces $\mathbb {CP}^{m_1} \times \cdots \times \mathbb {CP}^{m_s}$?
\end{qu}
We want to or should close this section with the following famous theorem. If we consider a \emph{nonsingular} surface which is homotopic to $\mathbb {CP}^2$, we have the following famous theorem due to Shing-Tung Yau \cite{Yau}, which is an answer to Severi's old problem\footnote{Francesco Severi posed the question of whether there was a complex surface homeomorphic to $\mathbb {CP}^2$ but not biholomorphic to it. So, S.-T. Yau gave a negative answer even in a weaker version of ``homeomorphic" being replaced by ``homotopic".} \cite{Severi} :
{}
\begin{thm}[S.-T. Yau] 
Any complex surface which is homotopic to $\mathbb {CP}^2$ is biholomorphic to $\mathbb {CP}^2$.
\end{thm}
\begin{rem}
It is quite natural to ask if a similar statement holds for $\mathbb {CP}^n$ with $n \geq 3$. As far as the author knows by literature search, there seems to be no result available for this naive question.
\end{rem}
\section{Fundamental properties of rationally elliptic spaces}
In this section we recall some fundamental properties of rationally elliptic spaces for later use of them.

\subsection {S. Halperin's Theorems}
$$\chi^{\pi}(X):=\sum_k (-1)^k \dim (\pi_k(X)\otimes \Q) = \dim (\pi_{\op{even}}(X)\otimes \Q) - \dim (\pi_{\op{odd}}(X)\otimes \Q)$$
is called \emph{homotopy Euler--Poincar\'e characteristic}, which is surely a homotopy version of Euler--Poincar\'e characteristic:
$$\chi(X)=\sum_k (-1)^k \dim (H_k(X;\Q) = \dim H_{\op{even}}(X;\Q) - \dim H_{\op{odd}}(X;\Q).$$
First let us look at those of $S^n$ and $\mathbb {CP}^n$ (see Examples \ref{sphere} and \ref{proj} above).
$$\text{$\chi^{\pi}(S^{2k})=(-1)^{2k}+(-1)^{4k-1}=0$ \quad \text{and} \quad $\chi(S^{2k})=1+(-1)^{2k}=2 >0$}$$
$$\text{$\chi^{\pi}(S^{2k+1})=(-1)^{2k+1}=-1<0$ \quad \text{and} \quad $\chi(S^{2k+1})=1+(-1)^{2k+1}=0$.}$$
$$\text{$\chi^{\pi}(\mathbb {CP}^n)=(-1)^{2}+(-1)^{2n+1}=0$ \quad \text{and} \quad $\chi(\mathbb {CP}^n)=1+\sum_{i=1}^n (-1)^{2i}=1+n>0.$} $$
Hence we have
$$\text{$\chi^{\pi}(S^{n_1} \times \cdots \times S^{n_j} \times \mathbb {CP}^{m_1} \times \cdots \times \mathbb {CP}^{m_s})\leq 0$ and} \hspace{4cm}$$
$$\hspace{6cm}\chi(S^{n_1} \times \cdots \times S^{n_j} \times \mathbb {CP}^{m_1} \times \cdots \times \mathbb {CP}^{m_s}) \geq 0$$
$$\text{$\chi^{\pi}(S^{2n_1} \times \cdots \times S^{2n_j} \times \mathbb {CP}^{m_1} \times \cdots \times \mathbb {CP}^{m_s}) =0$ and }\hspace{4cm}$$
$$\hspace{6cm} \chi(S^{2n_1} \times \cdots \times S^{2n_j} \times \mathbb {CP}^{m_1} \times \cdots \times \mathbb {CP}^{m_s})>0.$$
As observed above, $S^{n_1} \times \cdots \times S^{n_j} \times \mathbb {CP}^{m_1} \times \cdots \times \mathbb {CP}^{m_s}$ is rationally elliptic. It turns out that this property ``$\chi^{\pi}(X) \leq 0$ and $\chi(X) \geq 0$" always holds for any rationally elliptic space $X$, as proved by 
S. Halperin \cite{Hal}: 
\begin{thm}\cite[Theorem 1, p.174]{Hal}\label{hal-1}
Let $X$ be a rationally elliptic space. Then $\chi^{\pi}(X) \leq 0$ and $\chi(X) \geq 0$. Moreover, the following are equivalent:
\begin{enumerate}
\item  $\chi^{\pi}(X)=0$, 
\item $\chi(X) >0,$
\item $H_{odd}(X)\otimes \Q=0$.
\end{enumerate}
\end{thm}
\begin{rem}
Thus an affirmative solution of Bott conjecture implies Hopf conjecture.
\end{rem}
\begin{rem} The equivalence of (1), (2) and (3) above was posed as a question in Dennis Sullivan's famous paper \cite{Sul}.
 In \cite{Hes} Kathryn Hess  wrote that Halperin's paper was \emph{``the first major paper on the structure and properties of Sullivan models after the work of Sullivan himself"}.
 \end{rem}
\begin{rem}
According to \cite{Hal}, a special case of Theorem \ref{hal-1} was already solved by Henri Cartan \cite{Car} and the proof of Theorem \ref{hal-1} is reduced to this special case.
\end{rem}
Next let us look at 
$$\text{$\op{dim}(X)$, \quad Poincar\'e polynomial $P_X(t)= \sum_k  \dim H_k(X;\mathbb Q) \, t^k$ \quad and \quad $\chi(X)$}$$
 of the above examples (cf. Examples \ref{sphere} and \ref{proj}). First we recall the rational homotopy groups of $S^{2k}$ and $\mathbb {CP}^n$ \underline{in the following forms}:

$$\pi_i(S^{2k}) \otimes \Q 
=\begin{cases} 
 \Q & \, i=2k, 2(2k)-1\\
\, 0  & \,  i\not =2k, 2(2k)-1
\end{cases}, \, \, 
\pi_k(\mathbb {CP}^n)\otimes \mathbb Q = \displaystyle 
\begin{cases}
\mathbb Q \quad \text{for $k=2, 2n+1 =2(n+1)-1$,}\\
0 \quad \text{for $k\not = 2, 2n+1=2(n+1)-1$,}\\
\end{cases}
$$
\begin{enumerate}
\item $$\op{dim}(S^{2k})=2k= 2(2k)-1 - (2k-1), P_{S^{2k}}(t)=1+t^{2k} = \frac{1-t^{2(2k)}}{1-t^{2k}}, \chi(S^{2k})= 2 = \frac{2k}{k}, $$
\item $$\op{dim}(\mathbb {CP}^n)=2n= 2(n+1)-1 - (2\cdot 1-1), P_{\mathbb {CP}^n}(t) = 1+\sum_{i=1}^n t^{2i} =\frac{1-t^{2(n+1)}}{1-t^{2\cdot 1}},$$
$$\chi(\mathbb {CP}^n)=n+1= \frac{n+1}{1}.$$
\item 
\begin{align*}
\op{dim} & \left (\prod_{i=1}^k  S^{2n_i} \times \prod_{j=1}^s \mathbb {CP}^{m_j} \right)  =\sum_{i=1}^k 2n_i + \sum_{j=1}^s 2m_j \\
&= \sum_{i=1}^k \{2(2n_i)-1\} + \sum_{j=1}^s \{2(m_j+1)-1 \} - \left (\sum_{i=1}^k \{2\cdot n_i-1\} + \sum_{j=1}^s \{2\cdot 1-1\} \right ).
\end{align*}
\begin{align*}
 & P_{\prod_{i=1}^k S^{2n_i} \times \prod_{j=1}^s \mathbb {CP}^{m_j}} (t)  =\prod_{i=1}^k P_{S^{2n_i}}(t) \prod_{j=1}^s P_{\mathbb {CP}^{m_j}}(t) \\
& \qquad \qquad = \prod_{i=1}^k \frac{1-t^{2(2n_i)}}{1-t^{2n_i}}\prod_{j=1}^s \frac{1-t^{2(m_j+1)}}{1-t^2} 
= \frac{\prod_{i=1}^k (1-t^{2(2n_i)}) \prod_{j=1}^s (1-t^{2(m_j+1)})}{\prod_{i=1}^k (1-t^{2n_i}) \cdot (1-t^{2\cdot 1})^s},
\end{align*}
$$\chi\left (\prod_{i=1}^k S^{2n_i} \times \prod_{j=1}^s \mathbb {CP}^{m_j} \right)=2^k \prod_{j=1}^s (m_j+1) = \frac{\prod_{i=1}^k 2n_i \prod_{j=1}^s (m_j+1)}{\prod_{i=1}^k n_i \prod_{j=1}^s 1}.$$
\end{enumerate}

In fact, these things hold for any rationally elliptic space, which was proved also by S. Halperin \cite{Hal}, as shown below. Let $X$ be rationally elliptic. 

Let $y_1, \cdots, y_q$ be a basis of $\pi_{\op{odd}}(X)\otimes \Q$ ($\dim(\pi_{\op{odd}}(X)\otimes \Q)=q$) and $x_1, \cdots, x_r$ be a basis of $\pi_{\op{even}}(X)\otimes \Q$ ($\dim(\pi_{\op{even}}(X)\otimes \Q)=r$). If $y_j \in \pi_{2b_j-1}(X)\otimes \Q$ and $x_i \in \pi_{2a_i}(X)\otimes \Q$, $2b_j-1$ and $2a_i$ are called \emph{degrees} of $y_j$ and $x_j$. 
$(b_1, \cdots, b_q)$ and $(a_1, \cdots, a_r)$ are respectively called \emph{$b$-exponents} and \emph{$a$-exponetns} of $X$
(note that they are called ``odd" exponents and ``even" exponents in F\'elix--Halperin--Thomas's book \cite[\S 32 Elliptic spaces, Definition, p.441]{FHT}).
E.g., for $X=\prod_{i=1}^k S^{2n_i} \times \prod_{j=1}^s \mathbb {CP}^{m_j}$, we have
$\pi_{2\cdot 2n_1-1}(X)\otimes \mathbb Q= \cdots =\pi_{2\cdot 2n_k-1}(X)\otimes \mathbb Q = \pi_{2(m_1+1)-1}(X)\otimes \mathbb Q= \cdots =\pi_{2(m_s+1)-1}(X)\otimes \mathbb Q =\mathbb Q$ and 
$\pi_{2n_1}(X)\otimes \mathbb Q= \cdots =\pi_{2n_k}(X)\otimes \mathbb Q = \mathbb Q, \quad  \pi_{2\cdot 1}(X)\otimes \mathbb Q= \cdots =\underbrace {\mathbb Q \oplus \cdots \mathbb Q}_{s}$.
$(2n_1,\cdots,2n_k,m_1+1,\cdots,m_s+1)$ are $b$-exponents and 
$(n_1,\cdots,n_k,\underbrace{1,\cdots,1}_s)$ are $a$-exponents of $X$.

The largest integer $n_X$ such that $H_{n_X}(X;\Q) \not =0$ is called \emph{formal dimension} of $X$.
\begin{thm}\label{Halp-2}\cite[Theorem 3', p.188, Corollary 2, p.198]{Hal} 
\label{str}  Let the symbols be as above.
\begin{enumerate}
\item \label{b} 

$n_X=
\sum_{j=1}^q (2b_j-1) -\sum_{i=1}^r (2a_i-1)$ 
\item \label{poincare} Betti numbers $\beta_i =\dim H_i(X;\mathbb Q)$ satisfy Poincar\'e duality, i.e., $\beta_i = \beta_{n_X -i}$.
\item \label{6} In the case when $\chi^{\pi}(X)=q-r=0$ (thus $\chi(X)>0$), i.e., $q=r$, Poincar\'e polynomial $P_X(t)= \sum_k  \dim H_k(X;\mathbb Q) \, t^k$ of $X$ is expressed by $b$-exponents and $a$-exponents:
$$P_X(t)= \frac{\prod_{i=1}^q (1 - t^{2b_i})}{\prod_{i=1}^q (1 - t^{2a_i})}.$$ 
In particular,
\begin{equation}\label{b/a}
\chi(X)= P_X(-1) = P_X(1)= \dim (H_*(X)\otimes \mathbb Q) =\frac{ \prod_{i=1}^q b_i}{ \prod_{i=1}^q a_i}.
\end{equation} 
\end{enumerate}
\end{thm}
\begin{rem}
Here we just remark that (\ref{b/a}) is due to the following modification:
\begin{align*}
P_X(t)= \frac{\prod_{i=1}^q (1 - t^{2b_i})}{\prod_{i=1}^q (1 - t^{2a_i})} &= \frac{ \prod_{i=1}^q (1 - (t^2)^{b_i})}{\prod_{i=1}^q (1 - (t^2)^{a_i})} \\
& = \frac{(1-t^2)^q \prod_{i=1}^q (1 + t^2 + \cdots + (t^2)^{b_i-1})}{(1-t^2)^q \prod_{i=1}^q (1 + t^2 + \cdots + (t^2)^{a_i-1})}\\
&=\frac{\prod_{i=1}^q (1 + t^2 + \cdots + (t^2)^{b_i-1})}{\prod_{i=1}^q (1 + t^2 + \cdots + (t^2)^{a_i-1})}.\end{align*}
\end{rem}
\subsection {J. B. Friedlander--S. Halperin's Theorems}
In this section we recall some fundamental results of Friedlander--Halperin \cite{FH}.
\begin{defn}\cite[Introduction, \emph{Definition}, p.117-118]{FH} Let $B=(b_1, \cdots, b_q)$ and $A=(a_1, \cdots, a_r)$ be sequences of positive integers. 
\begin{enumerate}
\item We say $(B;A)$ satisfies \emph{strong arithmetic condition} (abbr. S.A.C.) if for every subsequence $A^*$ of $A$ of \underline{length $s$} ($1 \leq s \leq r$) there exists \underline{at least $s$ elements} $b_j$'s of $B$ such that
$$b_j = \sum_{a_i \in A^*} \gamma_{ij}a_i$$
where $\gamma_{ij}$ is \emph{a non-negative integer} such that $\sum_{a_i \in A^*} \gamma_{ij} \geq 2.$
\item If $\sum_{a_i \in A^*} \gamma_{ij} \geq 2$ is not required, then we say that $(B;A)$ satisfies \emph{arithmetic condition} (abbr. A.C.).
\end{enumerate}
Thus, in both cases, it is \emph{necessary} that $r \leq q$ (by considering $s=r$).
\end{defn}
{} 
\begin{ex}
\begin{enumerate}
\item  $B=(3,4,6)$, $A=(2,3,4)$. $(B;A)$ satisfies A.C., but not S.A.C..
\item $B=(4,6,8)$, $A=(2,3,4)$. $(B;A)$ satisfies S.A.C. (hence A.C.).
\item $B=(3,4,5,5,8)$, $A=(2,3,4)$. $(B;A)$ satisfies A.C., but not S.A.C..
\item $B=(3,4,5,6,8)$, $A=(2,3,4)$. $(B;A)$ satisfies S.A.C. (hence A.C.).
\item $B=(3,5,7)$, $A=(2,4)$. $(B;A)$ does not satisfy A.C. (hence not S.A.C.).
\end{enumerate}
\end{ex}
\begin{thm}\cite[Theorem 1, p.118]{FH} \label{sac}
Let $B=(b_1, \cdots, b_q)$ and $A=(a_1, \cdots, a_r)$ be finite sequences of positive integers. The following are equivalent:
\begin{enumerate}
\item $(B;A)$ satisfies S.A.C.
\item $B$ and $A$ are $b$-exponents and $a$-exponents of a rationally elliptic space $X$.
\end{enumerate}
\end{thm}
The above Theorem \ref{sac} is equivalent to Theorem \ref{ag-theo} below:

Let $R=\frak K [u_1, \cdots, u_r]$ the ring of polynomials in $r$ variables $u_i$ of degree $a_i$ over an infinite field $\frak K$. Let
$$\Phi_i :=\{\sigma_{ij}\}_{j=1, \cdots, \ell_i} , i=1, \cdots, q$$
be families of non-linear monomials $\sigma_{ij}$ of degree $b_i$ in the variables $u_1, \cdots, u_r$. Then $(B,A)$ satisfies S.A.C. if and only if the families $\Phi_1, \cdots, \Phi_q$ satisfies the following P.C.:

\begin{defn}\cite[\emph{Definition}, p.119]{FH}
The families $\Phi_1, \cdots, \Phi_q$ satisfy P.C. (polynomial condition) if and only if for each $s$ and for each set of $s$ variables $u_{i_1}, \cdots, u_{i_s}$ there are at least $s$ families $\Phi_{m_1}, \cdots, \Phi_{m_s}$ containing a non-linear monomial in $\frak K [u_{i_1}, \cdots, u_{i_s}]$.
\end{defn}

\begin{thm}\cite[Theorem 3, p.119]{FH} \label{ag-theo} Assume that $\Phi_1, \cdots, \Phi_q$ are sets of monomials as above. $\Phi_1, \cdots, \Phi_q$ satisfy P.C. if and only if there are polynomials $f_1, \cdots, f_q$ of the form
$$f_i = \sum_{j=1}^{\ell_i} c_{ij}\sigma_{ij}, \quad c_{ij} \in \frak K, \sigma_{ij} \in  \Phi_i \, \, (1 \leq j \leq \ell_i)$$
such that $\op{dim}_{\frak K} \left (\displaystyle \frac{\frak K [u_1, \cdots, u_r]}{(f_1, \cdots, f_q)} \right ) < \infty.$
\end{thm}
The major part of Friedlander--Halperin's paper \cite{FH} is devoted to the proof of this ``algebro-geometric" Theorem \ref{ag-theo}.
\, 
{} 
\begin{cor} \label{ba} \cite[2.5 Lemma, p.121]{FH}
\label{b>2a} If $B=(b_1, b_2, \cdots, b_q)$; $b_1 \geq b_2 \geq \cdots \geq b_q$,  and $A=(a_1, a_2, \cdots, a_r)$; $a_1 \geq a_2 \geq \cdots \geq a_r$. If $(B;A)$ satisfies S.A.C, then $b_i \geq 2a_i \quad (1 \leq i \leq r).$
\end{cor}

{} 
\begin{cor}\cite[1.3 Corollary, p.118, 2.6 Proposition, p.121]{FH} 
\begin{enumerate}
\item $n_X \geq q+r = \dim (\pi_*(X) \otimes \Q)$.
\item $n_X \geq  \sum_{j=1}^q b_j.$
 \item $2n_X-1 \geq  \sum_{j=1}^q(2b_j-1).$  (Equality holds for $X=S^{2n}$; $2(2n)-1=4n-1.$)
\item $n_X \geq   \sum_{i=1}^r 2a_i.$ (Equality holds for $X=S^{2n}$; $n_X=2n=2a$)
\end{enumerate}
\end{cor}
\begin{proof}
All these inequalities will be used later. So, for the sake of convenience of the reader, the proof is given below.
\begin{align*}
\quad (1) \quad  n_X =\sum_{j=1}^q (2b_j-1) -\sum_{i=1}^r (2a_i-1) {}
& = \sum_{j=1}^q (b_j-1) + \sum_{j=1}^q b_j -\sum_{i=1}^r 2a_i+r \\
&\geq \sum_{j=1}^q (b_j-1) + \sum_{j=1}^r b_j -\sum_{i=1}^r 2a_i+r \quad \text{(since $q \geq r$)}\\
& \geq q + \sum_{i=1}^r (b_i -2a_i) +r \quad \text{(since $b_j \geq 2$)}\\
& \geq q+r \quad \text{(since $b_i \geq 2a_i$)}
\end{align*}
\begin{align*}
\quad (2) \quad  n_X =\sum_{j=1}^q (2b_j-1) -\sum_{i=1}^r (2a_i-1) 
& = \sum_{j=1}^q b_j +\sum_{j=1}^q (b_j-1) -\sum_{i=1}^r (2a_i-1) \\
&\geq \sum_{j=1}^q b_j +\sum_{j=1}^r (b_j-1)-\sum_{i=1}^r (2a_i-1)  \quad \text{(since $q \geq r$)}\\
& = \sum_{j=1}^q b_j + \sum_{i=1}^r (b_i -2a_i) \geq \sum_{j=1}^q b_j \quad \text{(since $b_i \geq 2a_i$)}
\end{align*}

(3)
$2n_X \geq  \sum_{j=1}^q 2b_j  \Longrightarrow 2n_X -q \geq \sum_{j=1}^q (2b_j-1) 
\Longrightarrow 2n_X-1 \geq \sum_{j=1}^q(2b_j-1)$.

(4)
$n_X \geq  \sum_{j=1}^q b_j \Longrightarrow n_X \geq \sum_{j=1}^r b_j \Longrightarrow n_X \geq \sum_{i=1}^r 2a_i$.
\end{proof}
\begin{rem} In fact, as one can see, the inequalities in (1), (3) and (4) can be made sharper ones as follows:
\begin{enumerate}
\item $n_X \geq 3q -r \geq 2q \geq \dim (\pi_*(X) \otimes \Q)$. It suffices to show $n_X \geq 3q -r$, whose proof is a slight revision of the above proof. Indeed,
\begin{align*}
 n_X =\sum_{j=1}^q (2b_j-1) -\sum_{i=1}^r (2a_i-1)
& = \sum_{j=1}^q (b_j-1) + \sum_{j=1}^q b_j -\sum_{i=1}^r 2a_i+r \\
&\geq \sum_{j=1}^q (b_j-1) + \sum_{j=r+1}^q b_j + \sum_{j=1}^r b_j -\sum_{i=1}^r 2a_i+r \\
& \geq q + 2(q-r) + \sum_{i=1}^r (b_i -2a_i) +r \quad \text{(since $b_j \geq 2$)}\\
& \geq q+ 2(q-r)+ r =3q-r. \quad \text{(since $b_i \geq 2a_i$)}
\end{align*}
\item[(3)] $2n_X -q \geq \sum_{j=1}^q (2b_j-1)$ is sharper than $2n_X-1 \geq \sum_{j=1}^q(2b_j-1)$.
\item[(4)] $\sum_{j=1}^q b_j = \sum_{j=1}^r b_j + \sum_{j=r+1}^q b_j \geq \sum_{i=1}^r 2a_i + 2(q-r) = \sum_{i=1}^r 2a_i -2\chi^{\pi}(X).$
Hence  $n_X \geq \sum_{i=1}^r 2a_i -2\chi^{\pi}(X)$ is sharper than $n_X \geq \sum_{i=1}^r 2a_i$. Note that $\chi^{\pi}(X) \leq 0$. 
\end{enumerate}
\end{rem}
\begin{pro}(cf. \cite[Proposition 2.2]{CatMil}) 
If $X$ is rationally elliptic and $\chi(X)>0$ (such a space is called positively elliptic), then the Hilali conjecture holds.
\end{pro}
\begin{proof}
Here we give a much easier proof different from that of \cite[Proposition 2.2]{CatMil}, \emph{as a corollary} of Friedlander--Halperin's results.
Since $\chi(X)>0$, {} $\chi^{\pi}(X)=\dim (\pi_{\op{even}}(X)\otimes \Q) - \dim (\pi_{\op{odd}}(X)\otimes \Q) = r-q=0$, i.e., $r=q$. Hence
$$\dim (H_*(X)\otimes \mathbb Q) =\frac{\prod_{i=1}^q b_i}{\prod_{i=1}^q a_i} {} \geq \frac{\prod_{i=1}^q 2a_i}{\prod_{i=1}^q a_i} {} = \prod_{i=1}^q 2 =2^q.$$
$\dim (\pi_*(X) \otimes \Q) =\dim (\pi_{\op{even}}(X)\otimes \Q) + \dim (\pi_{\op{odd}}(X)\otimes \Q)  =2q  \leq 2^q \leq \dim (H_*(X)\otimes \mathbb Q).$
\end{proof}
In the case of $q >r$, an explicit description of the Poincar\'e polynomial $P_X(t)$ is still open (as far as the author knows). However we do have the following bound for $P_X(t)$:
\begin{thm}\label{P<Q} \cite[2.8 Proposition, p.121, \emph{Proof of Proposition 2.8}, pp.131-132]{FH}
\begin{equation}\label{Q}
Q_X(t):=\frac{\prod_{i=1}^q (1 - t^{2b_i})}{(1-t)^{q-r}\prod_{j=1}^r (1 - t^{2a_j})} = 1 + \cdots + c_mt^m + \cdots + c_{n_X}t^{n_X}
\end{equation}
is such that each $c_m$ is non-negative and $\op{dim} H_m(X;\mathbb Q) \leq c_m$ for each $m$. So, 
$$\text{$P_X(t) \leq Q_X(t)$, in particular, $\op{dim} H_*(X;\mathbb Q) \leq Q_X(1)$.}$$ 
\end{thm} 
The difficult part is $\op{dim} H_m(X;\mathbb Q) \leq c_m$, for which they use \emph{transcendence degree or Krull dimension and Macaulay's theorem} in commutative ring theory.
\begin{thm}\cite[2.9 Corollary, p.121]{FH}
$$\dim H_*(X;\Q) \leq 2^{q-r} \frac{\prod_{j=1}^q b_j}{ \prod_{i=1}^r a_i} \leq (2n_X)^{n_X}$$
\end{thm} 
\begin{proof}
We write down a proof for the sake of convenience of the reader.
\begin{align*}
Q_X(t)=\frac{\prod_{i=1}^q (1 - t^{2b_i})}{(1-t)^{q-r}\prod_{i=1}^r (1 - t^{2a_i})}
 &= \frac{(1-t^2)^q \prod_{j=1}^q (1+t^2 + \cdots +(t^2)^{b_j-1})} {(1-t)^{q-r} (1-t^2)^r \prod_{i=1}^r (1+t^2 + \cdots +(t^2)^{a_i-1})} \\
&= \frac{(1-t^2)^{q-r} \prod_{j=1}^q (1+t^2 + \cdots +(t^2)^{b_j-1})} {(1-t)^{q-r}\prod_{i=1}^r (1+t^2 + \cdots +(t^2)^{a_i-1})}\\
&=\frac{(1+t)^{q-r} \prod_{j=1}^q (1+t^2 + \cdots +(t^2)^{b_j-1})} {\prod_{i=1}^r (1+t^2 + \cdots +(t^2)^{a_i-1})},
\end{align*}
which implies that $Q_X(1) = 2^{q-r} \frac{\prod_{j=1}^q b_j}{ \prod_{i=1}^r a_i}.$ Or we do the following modification:
\begin{align*}
Q_X(t) =\frac{\prod_{i=1}^q (1 - t^{2b_i})}{(1-t)^{q-r}\prod_{i=1}^r (1 - t^{2a_i})}
& = \frac{(1-t)^q \prod_{j=1}^q (1+t + \cdots +t^{2b_j-1})} {(1-t)^{q-r} (1-t)^r \prod_{i=1}^r (1+t + \cdots +t^{2a_i-1})}\\
&= \frac{\prod_{j=1}^q (1+t + \cdots +t^{2b_j-1})} {\prod_{i=1}^r (1+t + \cdots +t^{2a_i-1})}
\end{align*}
which also implies $Q_X(1) = \frac{\prod_{j=1}^q 2b_j}{ \prod_{i=1}^r 2a_i} = 2^{q-r} \frac{\prod_{j=1}^q b_j}{ \prod_{i=1}^r a_i}.$
Hence
$$\dim H_*(X;\Q) = P_X(1) \leq Q_X(1) \leq  2^{q-r} \frac{\prod_{j=1}^q b_j}{ \prod_{i=1}^r a_i}.$$
As to the second inequality 
$2^{q-r} \frac{\prod_{j=1}^q b_j}{ \prod_{i=1}^r a_i} \leq (2n_X)^{n_X}$, 
their proof is not written in \cite{FH}, but it must be as follows:
$$2^{q-r} \frac{\prod_{j=1}^q b_j}{ \prod_{i=1}^r a_i} \leq 2^{n_X} \prod_{j=1}^q n_X \leq 2^{n_X} (n_X)^{n_X} = (2n_X)^{n_X}.$$
Because $n_X \geq q+r$ and $n_X \geq \sum_{j=1}^q b_j$.
\end{proof}
In fact, a sharper and better one holds:
\begin{equation}\label{2^n}
2^{q-r} \frac{\prod_{j=1}^q b_j}{ \prod_{i=1}^r a_i} \leq 2^{n_X},
\end{equation}
thus we get
\begin{equation}\label{dim-2^n}
\dim H_*(X;\Q) \leq 2^{n_X}
\end{equation}
 So, Gromov conjecture holds for $\mathbb F=\mathbb Q$.
\begin{proof}[Proof of (\ref{2^n}) (by Halperin \cite{Hal2})]
$$\dim H_*(X;\Q) \leq 2^{q-r} \frac{\prod_{j=1}^q b_j}{ \prod_{i=1}^r a_i} = {} \frac{\prod_{j=1}^q 2b_j}{ \prod_{i=1}^r 2a_i} {} <\prod_{j=1}^q 2b_j \leq {} \prod_{j=1}^q 2^{b_j}= {} 2^{\sum_{j=1}^q b_j} \leq 2^{n_X}.$$
\end{proof}
\begin{rem} The above inequality (\ref{dim-2^n}) is in \cite[Theorem 2.75, p.85]{FOT}, whose references are Halperin's paper \cite{Hal} and F\'elix--Halperin--Thomas's book \cite{FHT}, but such a formula is not written anywhere in these two references. So, we understand that $(2n_X)^{n_X}$ in the above inequality $2^{q-r} \frac{\prod_{j=1}^q b_j}{ \prod_{i=1}^r a_i} \leq (2n_X)^{n_X}$ 
was  \emph{not a misprint} of $2^{n_X}$.
\end{rem}
\begin{rem}
It might be interesting to obtain a better or sharper inequality for the inequalities $\dim H_*(X;\Q) \leq 2^{q-r} \frac{\prod_{j=1}^q b_j}{ \prod_{i=1}^r a_i} \leq 2^{n_X}$.
\begin{enumerate}
\item For example, a bit sharper one is the following. If we use the descending order for $b$-exponents and $a$-exponents as in Corollary \ref{ba}, $b_1$ is the  maximum of the $b$-exponents and $a_r$ is the minimum of the $a$-exponents. Hence we have the following inequalities
\begin{equation}\label{b_1}
\dim H_*(X;\Q) \leq 2^{q-r}\frac{\prod_{j=1}^q b_j}{ \prod_{i=1}^r a_i} \leq 2^{q-r}\frac{(b_1)^q}{(a_r)^r} = \frac{(2b_1)^q}{(2a_r)^r} \leq \frac{(2b_1)^q}{2^r}=2^{q-r}(b_1)^q.
\end{equation}
E.g., if $X = \mathbb {CP}^n$, then we have $\dim H_*( \mathbb {CP}^n;\Q) \leq 2^{q-r}(b_1)^q= n+1.$
\item If $a_r \geq 2$, then we get a bit better one:
\begin{equation}
\dim H_*(X;\Q) \leq  2^{q-r} \frac{\prod_{j=1}^q b_j}{ \prod_{i=1}^r a_i}  \leq \frac{(2b_1)^q}{(2a_r)^r} = 2^{q-r}\frac{(b_1)^q}{(a_r)^r}
\end{equation}
E.g., if $X = \mathbb S^{2n}$, then we have $\dim H_*(S^{2n};\Q) \leq \displaystyle 2^{q-r}\frac{(b_1)^q}{(a_r)^r}= \frac{2n}{n}=2.$
\item If we use the following formula for geometric and arithmetic means
$$\left (\prod_{i=1}^n x_i \right) ^{\frac{1}{n}} \leq \frac{1}{n} \sum_{i=1}^n x_i.$$
we can get the following inequality:
\begin{equation}
\dim H_*(X;\Q) \leq  2^{q-r} \frac{\prod_{j=1}^q b_j}{ \prod_{i=1}^r a_i} \leq \left (\frac{2n_X}{q} \right )^q.
\end{equation}
Indeed,
$$2^{q-r} \frac{\prod_{j=1}^q b_j}{ \prod_{i=1}^r a_i} \leq 2^{q-r} \left (\frac{1}{q}\sum_{j=1}^qb_j \right )^q \leq \left(\frac{2}{q}\right)^q \left (\sum_{j=1}^qb_j \right )^q \leq \left (\frac{2n_X}{q} \right )^q.$$
E.g., $\dim H_*( \mathbb {CP}^n;\Q) \leq 2(2n)=4n$ and $\dim H_*(S^{2n};\Q) \leq 2(2n)=4n$.
\item In fact, the proof of (\ref{2^n}) can be modified to get a better inequality:
$$\dim H_*(X;\Q) \leq 2^{q-r} \frac{\prod_{j=1}^q b_j}{ \prod_{i=1}^r a_i} = \frac{\prod_{j=1}^q 2b_j}{ \prod_{i=1}^r 2a_i} \leq \frac{\prod_{j=1}^q 2b_j}{2^r} \leq \frac{\prod_{j=1}^q 2^{b_j}}{2^r} = \frac{2^{\sum_{j=1}^q b_j}}{2^r} \leq 2^{n_X-r}.$$
\end{enumerate}
\end{rem}
\begin{rem}
In \cite[Theorem 1]{Pa} Andrey V. Pavlov showed the following inequalities:
\begin{equation}\label{n-m}
\dim H_k(X;\Q) \leq \binom{n_X}{m},
\end{equation}
\begin{equation}\label{n.m.1}
\dim H_k(X;\Q) \leq \sum_{k+2\ell=m}\binom{q-r}{k} \binom{p}{\ell}.
\end{equation}
(\ref{n-m}) can be shown by $c_m \leq \binom{n_X}{m}$, where $c_m$ is the coefficient of $t^m$ of the polynomial $Q_X(t)$ (see \ref{Q}). A crucial key of Pavlov's theorem is that the roots of the polynomial $Q_X(t)$ are roots of the unity. Clearly (\ref{n-m}) also implies $\dim H_*(X;\Q) \leq 2^{n_X}$. In fact, (\ref{n.m.1}) implies the following inequality \cite[Corollary ]{Pa}:
\begin{equation*}\label{n.m}
\dim H_m(X;\Q) \leq \frac{1}{2}\binom{n_X}{m} (m \not =0, n_X),
\end{equation*}
which implies $\dim H_*(X;\Q) \leq 2^{n_X-1}+1$.
\end{rem}
\begin{rem} If one finds a compact Riemannian manifold $M$ of dimension $n$ with non-negative sectional curvature or entire Grauert tube such that $2^{n-1}+1 <\dim H_*(X;\Q) \leq 2^n$, then this manifold would be a counterexample to Bott conjecture, although it still satisfies Gromov conjecture.
\end{rem}
Before closing this section, we want to revise Theorem \ref{P<Q} a bit. In this theorem we assume that $q>r$ and $r \geq 1$. For example, as in the case when $X=S^{2n+1}$ is an odd sphere, it can happen that $r=0$. In fact, for any rationally elliptic space $X$, Theorem \ref{P<Q} still holds in the following sense:
\begin{thm} For any rationally elliptic space $X$ with $b$-exponents $(b_1,\cdots, b_q)$ and $a$-exponents $(a_1,\cdots,a_r)$, we have
$$P_X(t) \leq \frac{\prod_{i=1}^q (1 - t^{2b_i})}{(1-t)^{q-r}\prod_{j=1}^r (1 - t^{2a_j})}.$$ 
Here, if $r=0$, then it is understood that $\prod_{j=1}^r (1 - t^{2a_j})=1$. 
\end{thm}
\begin{proof} First we should note that $q \geq r$ since $X$ is rationally elliptic, thus $\chi^{\pi}(X)=r-q \leq 0$.
\begin{enumerate}
\item In the case when $q=r$, the \emph{equality} holds as shown in Theorem \ref{Halp-2} above.
\item In the case when $q>r$ and $r \geq 1$, it is nothing but Theorem \ref{P<Q}.
\item It remains to see the case when $r=0$. We consider the product space $X \times S^2$. Then, since we have $S_2(S^2)\otimes \Q=\Q$ and $S_3(S^2)\otimes \Q=\Q$ with $3=2\cdot 2 -1$, we have that $b$-exponents and $a$-exponents of $X \times S^2$ are respectively $\{b_1, \cdots, b_q, 2\}$ and $\{1\}$. Therefore it follows from 
Theorem \ref{P<Q} that we have
\begin{equation}\label{eq-9}
P_{X \times S^2}(t) \leq \frac{\left \{ \prod_{i=1}^q (1 - t^{2b_i}) \right \} (1-t^{2\cdot2})}{(1-t)^{(q+1)-1} (1-t^{2\cdot 1})} = \frac{\left \{ \prod_{i=1}^q (1 - t^{2b_i}) \right \} (1+t^2)}{(1-t)^q}.
\end{equation}
Since $P_{X \times S^2}(t)=P_X(t) \times P_{S^2}(t)=P_X(t) \times (1+t^2)$, cancelling out the term $(1+t^2)$ of the both sides of (\ref{eq-9}), the above inequality becomes
\begin{equation}
P_X(t) \leq \frac{\prod_{i=1}^q (1 - t^{2b_i})}{(1-t)^q}.
\end{equation}
\end{enumerate}
\end{proof}
\begin{ex} For example, consider $X= S^{2b_1-1} \times \cdots \times S^{2b_q-1}$. Then we have
$$P_X(t) = \prod_{i=1}^q (1+ t^{2b_i-1}) < \frac{\prod_{i=1}^q (1 - t^{2b_i})}{(1-t)^q} = \prod_{i=1}^q (1+t + t^2 + \cdots + t^{2b_i-1}).$$
The equality does not hold since $b_i \geq 2$.
\end{ex}

\section {Rationally elliptic smooth toric varieties}
Indranil Biswas, Vicente Mu\~noz and Aniceto Murillo \cite{BMM} proved 
\begin{thm}[Biswas--Mu\~noz--Murillo] \label{bmm}
If a compact smooth troic variety $X$ is rationally elliptic, then its Poincar\'e polynomial $P_X(t)$ is equal to that of a product of complex projective spaces. I.e., if $dim X =n$, then $P_X(t) = P_{\mathbb {CP}^{m_1}}(t) \times \cdots \times P_{\mathbb {CP}^{m_k}}(t)$ with $n=m_1 + \cdots m_k$.
\end{thm}
Thus, since $X$ is rationally elliptic and $\chi(X) >0$, as a corollary we get
\begin{cor} The Hilali conjecture holds for a rationally elliptic smooth toric variety.
\end{cor}
In fact, the above theorem also holds for homotopical Poincar\'e polynomial $P^{\pi}_X(t)$, although we do not need it for the above corollary:
\begin{thm}[A. Libgober and S. Yokura \cite{LY1}]
If a compact smooth toric variety $X$ is rationally elliptic, then its homotopical Poincar\'e polynomial $P^{\pi}_X(t)$ is equal to that of a product of complex projective spaces in the above theorem. I.e., if $dim X =n$, then $P^{\pi}_X(t) = P^{\pi}_{\mathbb {CP}^{m_1}}(t) + \cdots + P^{\pi}_{\mathbb {CP}^{m_k}}(t)$ with $n=m_1 + \cdots m_k$.
\end{thm}

{}
\section {Rationally elliptic K\"ahler manifolds}
Jaume Amor\'os and I. Biswas \cite[Theorem 1.1]{AB} proved 
{} 
\begin{thm}[Amor\'os--Biswas]
A simply connected compact complex K\"ahler surface is rationally elliptic if and
only if it belongs to the following list:

(i) $\mathbb {CP}^2$, 

(ii) Hirzebruch surfaces (
a ruled surface over $\mathbb {CP}^1$) $\mathbb S_h= \mathbb P_{\mathbb {CP}^1}(\mathcal O \oplus \mathcal O(h))$ for $h \geq 0$, 

(iii) a simply connected general type surface $X$ with $q(X) = p_g(X) = 0$, $K_X^2 = 8$ and $c_2(X) = 4$.
\end{thm}
\begin{rem}\cite[Remark 3.1, p.1173]{AB}
$X$ in (iii) is called  
\emph{fake quadric}. Friedrich Hirzebruch asked whether 
fake quadrics exist. This question remains open. 
By Freedman's theorem  any 
fake quadric, \emph{if it exists}, is homeomorphic to quadric $\mathbb S_0 = \mathbb{CP}^1 \times \mathbb{CP}^1$ or Hirzebruch surface $\mathbb S_1= \mathbb P_{\mathbb {CP}^1}(\mathcal O \oplus \mathcal O(1))$ (the blow-up of $\mathbb {CP}^2$ at a point). 
\end{rem}
\begin{rem}
\begin{enumerate}
\item  Blow-up of $\mathbb {CP}^2$ at \underline{two points} is K\"ahler, since blow-up of a K\"ahler manifold at a point is still K\"ahler, but \emph{not} rationally elliptic, because it is not in the above list. So, \emph{blow-up sometimes destroys rational ellipticity}.
\item F. Hirzebruch (Math. Ann., 1951) showed that 
$$\text{$\mathbb S_h$ is diffeomorphic to $\mathbb S_k$ $\Longleftrightarrow h \equiv k (mod \, \, 2)$.}$$
 Thus a simply connected rationally elliptic compact K\"ahler surface is homeomorphic to $\mathbb {CP}^2$ or $\mathbb S_0=\mathbb{CP}^1 \times \mathbb {CP}^1$ or $\mathbb S_1=\mathbb {CP}^2 \# \overline {\mathbb {CP}^2}$ (the blow-up of $\mathbb {CP}^2$ at a point).
\item $\mathbb S_h$ is a $\mathbb {CP}^1=S^2$-bundle over $\mathbb {CP}^1=S^2$, $\pi_*(\mathbb S_h)=\pi_*(S^2)\oplus \pi_*(S^2)$ (by the long exact sequence and the existence of a section ($\infty$-section)). $H^*(\mathbb S_h) \cong H^*(\mathbb {CP}^1) \otimes H^*(\mathbb {CP}^1)$ since $\pi_1(\mathbb {CP}^1)=0$. So, $\dim (\pi_*(\mathbb S_h)\otimes \Q)=\dim H_*(\mathbb S_h;\Q)=4$. So, $\mathbb S_h$ satisfies Hilali conjecture.
\end{enumerate}
\end{rem}
$\mathbb {CP}^2$ is rationally elliptic and the blown-up $Blow_1(\mathbb {CP}^2)$ of $\mathbb {CP}^2$ at a point is still rationally elliptic, but the blown-up $Blow_2(\mathbb {CP}^2)$ of $\mathbb {CP}^2$ at two points is not rationally elliptic any more. Thus, it would be reasonable to pose the following question:
\begin{qu} 
\begin{enumerate}
\item Given a rationally elliptic complex manifold $M$, can one characterize properties of $M$ so that the blown-up of $M$ at a point is still rationally elliptic?
\item In general, given a rationally elliptic complex manifold $M$, can one characterize properties of $M$ so that the blown-up of $M$ at $k$ points is still rationally elliptic?
\end{enumerate}
\end{qu}
\begin{thm} \cite[Theorem 1.3]{AB}
If $X$ is a simply connected compact K\"ahler threefold which is rationally elliptic, then
the Hodge numbers satisfy that $h^{p,q}=0$ for $p \not =q$ and $h^{p,p}$ is one of the following:
\begin{enumerate}
\item $h^{0,0}=h^{1,1}=h^{2,2}=h^{3,3}=1$
\item $h^{0,0}=h^{3,3}=1$, $h^{1,1}=h^{2,2}=2.$
\item $h^{0,0}=h^{3,3}=1$, $h^{1,1}=h^{2,2}=3.$
\end{enumerate}
\end{thm}
\begin{rem}
The Poincar\'e polynomials of the above K\"ahler threefold  are respectively 
\begin{enumerate}
\item $1+t^2+t^4+t^6$, 
\item $1+2t^2+2t^4+t^6=(1+t^2)(1+t^2+t^4)$, 
\item $1+3t^2+3t^4+t^6=(1+t^2)^3$.
\end{enumerate}
which are respectively the same as the Poincar\'e polynomial of 
\begin{enumerate}
\item $\mathbb {CP}^3$, 
\item $\mathbb {CP}^1 \times \mathbb {CP}^2$,
\item $\mathbb {CP}^1 \times \mathbb {CP}^1 \times \mathbb {CP}^1$.
\end{enumerate}
\end{rem}
Yang Su and Jianqiang Yang \cite{SuYa} roved 
\begin{thm}
If $X$ is a simply connected compact K\"ahler fourfold which is rationally elliptic, then
the odd Betti numbers of $X$ are all zero (as above) and the Hodge numbers of $X$ is one of the following: $h^{p,q}=0$ if $p \not = q$ and $h^{p,p}$ is as follows:
\begin{enumerate}
\item $h^{0,0}=h^{1,1}=h^{2,2}=h^{3,3}=h^{4,4}=1$
\item $h^{0,0}=h^{1,1}=h^{3,3}=h^{4,4}=1$, $h^{2,2}=2.$
\item $h^{0,0}=h^{4,4}=1$, $h^{1,1}=h^{2,2}=h^{3,3}=2.$
\item $h^{0,0}=h^{4,4}=1$, $h^{1,1}=h^{3,3}=2$, $h^{2,2}=3.$
\item $h^{0,0}=h^{4,4}=1$, $h^{1,1}=h^{3,3}=3$, $h^{2,2}=4.$
\item $h^{0,0}=h^{4,4}=1$, $h^{1,1}=h^{3,3}=4$, $h^{2,2}=6.$
\end{enumerate}
\end{thm}
\begin{rem}
The Poincar\'e polynomials of the above K\"ahler threefold  are respectively 
\begin{enumerate}
\item $1+t^2+t^4+t^6+t^8$,
\item $1+t^2+2t^4+t^6+t^8=(1+t^4)(1+t^2+t^4)$,
\item $1+2t^2+2t^4+2t^6+t^8=(1+t^2)(1+t^2+t^4+t^6)$,
\item  $1+2t^2+3t^4+2t^6+t^8=(1+t^2+t^4)^2$,
\item  $1+3t^2+4t^4+3t^6+t^8=(1+t^2)^2(1+t^2+t^4)$,
\item $1+4t^2+6t^4+4t^6+t^8=(1+t^2)^4$,
\end{enumerate}
which are respectively the same as the Poincar\'e polynomial of 
\begin{enumerate}
\item $\mathbb {CP}^4$,
\item $S^4 \times \mathbb {CP}^2$ (note that $S^4$ is not K\"ahler), 
\item $\mathbb {CP}^1 \times \mathbb {CP}^3$, 
\item $\mathbb {CP}^2 \times \mathbb {CP}^2$, 
\item $\mathbb {CP}^1 \times \mathbb {CP}^1 \times \mathbb {CP}^2$,
\item $\mathbb {CP}^1 \times \mathbb {CP}^1 \times \mathbb {CP}^1 \times \mathbb {CP}^1$.
\end{enumerate}
\end{rem}
{}
\,
So, whatever the Hodge number $h^{p,p}$ is, the Hilali conjecture holds for a \emph{rationally elliptic} K\"ahler manifold $X$ of dimension $\leqq 4$, since $\chi(X)>0$.
How about higher dimension bigger than $4$? As far as the author knows, there is no work available.

However, speaking of K\"ahler manifolds, the following famous theorem for \emph{formality} (see \cite{DGMS} and \cite{FHT}) should be mentioned:
Pierre Deligne, Phillip Griffiths, John Morgan and D. Sullivan \cite{DGMS} proved the following 
\begin{thm}[Deligne--Griffiths--Morgan--Sullivan]
Any compact K\"ahler manifold is a \underline{formal} space.
\end{thm}
In fact, M. R. Hilali and My Ismail  Mamouni \cite[Theorem 2]{HM} proved the following
\begin{thm}[Hilali--Mamouni] The Hilali conjecture holds for any rationally elliptic formal space.
\end{thm}
Thus as a corollary, although we do knot know about Hodge numbers $h^{p,q}$ explicitly, we can have
\begin{cor}
The Hilali conjecture holds for any rationally elliptic K\"ahler manifold \underline{of any dimension}.
\end{cor}
Speaking of formal space, here are some examples of \emph{formal singular varieties}, for which see D. Chataur--J. Cirici's paper \cite[\S 4]{CC} 
\begin{enumerate}
\item Complete intersections with \emph{isolated singularities}.
\item Projective varieties with \emph {isolated singularities}, satisfying that there exists a resolution of 
singularities such that \emph{the exceptional divisor is smooth}.
\item Projective varieties whose singularities are only isolated \emph{ordinary multiple points}. 
\item \emph{Projective cones} over smooth projective varieties.
\end{enumerate}
Here is the following question:
\begin{qu}
Among the above list (1), (2), (3) and (4) of formal singular varieties, which are \underline{rationally elliptic?}
\end{qu}
Answering this question seems (to at least the author) to be not easy, as the following simple examples show:
\begin{ex} The projective cone over $\mathbb {CP}^1$ is $\mathbb {CP}^2$ and in general, the projective cone over $\mathbb {CP}^n$ is $\mathbb {CP}^{n+1}$.
\end{ex}
\begin{ex} Let us denote the projective cone over $X$ by $pc(X)$. The projective cone $pc \left ( \mathbb {CP}^1 \times \mathbb {CP}^1\right)$ of the quadric $\mathbb {CP}^1 \times \mathbb {CP}^1$ embedded into $\mathbb {CP}^3$, is \underline{not} rationally elliptic, although $\mathbb {CP}^1 \times \mathbb {CP}^1$ is rationally elliptic. This is due to the following. The Poincar\'e polynomial
$$P_{pc(\mathbb {CP}^1 \times \mathbb {CP}^1)}(t)=1+t^2(1+t^2)^2=1+t^2+2t^4+t^6,$$
which is by the following decomposition in the Grothendieck ring:
$$[pc \left ( \mathbb {CP}^1 \times \mathbb {CP}^1\right) ]= [c] + [\mathbb C][\mathbb {CP}^1 \times \mathbb {CP}^1]$$
Here $c$ is the cone point and $[\mathbb C]=[\mathbb {CP}^1 - c]$.
So, $1=\beta_2 \not = \beta_4=\beta_{6-2}=2$, i.e., the Poincar\'e duality does \underline{not} hold. I.e., $\op{dim}(\pi_*(pc(\mathbb {CP}^1 \times \mathbb {CP}^1)\otimes \mathbb Q)=\infty$.
Therefore \emph{the simple operation of taking the projective cone} also sometimes destroy rational ellipticity.
Here we note that this projective cone over the quadric is also considered in \cite[Example 2.2.4]{CatMig} as an example of a singular variety whose cohomology does not satisfy the Poincar\'e duality, but whose \emph{intersection homology} restores the Poincar\'e duality; $H^2(pc \left ( \mathbb {CP}^1 \times \mathbb {CP}^1\right);\Q)=\Q$, but $IH^2(pc \left ( \mathbb {CP}^1 \times \mathbb {CP}^1\right);\Q) =\Q \oplus \Q$. 
\end{ex}
\
\section {Formal dimension and Hilali conjecture}
At the moment we do not know \emph{a characterization of rationally elliptic algebraic varieties, smooth or singular}. However, \emph{as long as its complex dimension $n$ is less than or equal to $10$, i.e., $n_X =2n \leq 20$}, the Hilali conjecture always holds for such a variety, due to the following results:

\begin{thm}[M.R. Hilali and M.I. Mamouni \cite{HM2}]
If $n_X \leq 10$, the Hilali conjecture holds.
\end{thm}
This was extended up to $n_X=16$:
\begin{thm} [O. Nakamura and T.Yamaguchi \cite{NY}]
If $n_X \leq 16$, the Hilali conjecture holds.
\end{thm}
Furthermore this was extended up to $n_X=20$:
\begin{thm} [S. Cattalani and A. Milivojevic \cite{CatMil}]
If $n_X \leq 20$, the Hilali conjecture holds.
\end{thm}

As far as the author knows, we do not know whether the formal dimension $n_X$ has been extended \emph{bigger than 20}.

Before going to the next section, based on what we observed so far, we want to make the following naive conjecture:
\begin{con}\label{conjecture} If $X$ is a rationally elliptic complex algebraic variety, singular or non-singular, then its Poincar\'e polynomial $P_X(t)$ is the same as that of the product of even-dimensional spheres and complex projective spaces:
$$P_X(t) = \prod_{i=1}^k P_{S^{2n_i}}(t) \times  \prod_{j=1}^s P_{\mathbb {CP}^{m_j}}(t).$$ 
\end{con}
\begin{rem}
If Conjecture \ref{conjecture} is correct,  then, since $\chi(X)>0$, as a corollary we would get that \emph{the Hilali conjecture always holds for
any rationally elliptic complex algebraic variety, singular or non-singular.}
\end{rem}

\section{Hilali Conjecture modulo ``products"}

 \quad $\dim \pi_*(X)\otimes \Q$, $\dim H_*(X;\Q)$, $\chi^{\pi}(X)$, $\chi(X)$ are
special values of Poincar\'e polynomials:
$$P_X^{\pi}(t):= \sum_k \dim \pi_k(X;\Q) t^k, \quad P_X(t):= \sum_k \dim H_k(X;\Q) t^k.$$
$$\dim (\pi_*(X)\otimes \Q) =P_X^{\pi}(1), \quad \dim H_*(X;\Q) =P_X(1)$$
$$\chi^{\pi}(X)=P_X^{\pi}(-1), \quad \chi(X)=P_X(-1).$$
For the three ``distinguished values" for $t$, we have the following inequalities, although the third one is conjectural:
\begin{enumerate}
\item $t=-1$: $P_X^{\pi}(-1) < P_X(-1)$ (By Halperin's theorem). 
\item $t=0$: $0=P_X^{\pi}(0)<P_X(0)=1.$
\item $t=1$: $P_X^{\pi}(1) \leq P_X(1)$ (Hilali conjecture).
\end{enumerate}
How about the other values $t$ for comparing $P_X^{\pi}(t)$ and $P_X(t)$? Of course $P_X^{\pi}(t) \leq P_X(t)$ does not necessarily hold for some values $t$, as shown below.
\begin{ex}
$P_{S^2}^{\pi}(t)= t^2+t^3$ and $P_{S^2}(t)=1+t^2$. $P_{S^2}^{\pi}(1)= P_{S^2}(1)=2.$ 

\qquad  \qquad \qquad $P_{S^2}^{\pi}(2)=2^2+2^3=12 > P_{S^2}(2)=1+2^2=5.$
\end{ex}
``homotopy is additive" and ``homology is multiplicative", as we used above! Namely,
$$P_{X \times Y}^{\pi}(t) = P_X^{\pi}(t) + P_Y^{\pi}(t), \quad P_{X \times Y}(t) = P_X(t) \times P_Y(t).$$
In particular we have $P_{X^n}^{\pi}(1)=nP_X^{\pi}(1), P_{X^n}(1)= (P_X(1))^n.$
Here $X^n$ is the Cartesian product $X^n= \underbrace{X \times \cdots \times X}_n$.

In \cite{Yo} the author showed the following theorem:
\begin{thm}[Hilali Conjecture modulo ``products"] For any rationally elliptic space $X$, there exists an integer $n_0(X)$ such that for any integer $n \geq n_0(X)$
$$P_{X^n}^{\pi}(1) < P_{X^n}(1).$$
\end{thm}
This theorem was motivated by an elementary fact in Calculus: 
\begin{thm} If $|r|<1$, then $\displaystyle \lim_{n \to \infty}nr^n =0$.
\end{thm}
{} If $P_{X}(1)=1$, then $P^{\pi}_{X}(1)=0$ (by Theorem \ref{serre} (Serre Theorem) ), hence for any integer $n \geqq 1$ we have 
$$0= P^{\pi}_{X^n}(1) < P_{X^n}(1) = 1.$$
So let $P_X(1) \geqq 2$, thus $\frac{1}{P_X(1)} <1$. Hence we have 
$$\lim_{n \to \infty} n \Bigl (\frac{1}{P_X(1)} \Bigr)^n = 0.$$
Therefore, whatever the value $P^{\pi}_X(1)$ is, we have
$$\displaystyle \lim_{n \to \infty} n P^{\pi}_X(1)\Bigl (\frac{1}{P_X(1)} \Bigr)^n =  \displaystyle\lim_{n \to \infty} \frac{nP^{\pi}_X(1)}{(P_X(1))^n} = 0.$$
Hence there exists some integer $n_0$ such that for all $n \geqq n_0$ we have 
$\frac{P^{\pi}_{X^n}(1)}{P_{X^n}(1)}  = \frac{nP^{\pi}_X(1)}{(P_X(1))^n} < 1.$
Therefore there exists some integer $n_0$ such that for all $n \geqq n_0$
$$P^{\pi}_{X^n}(1) < P_{X^n}(1).$$
\section{``Hilali conjecture modulo products" for Poincar\'e polynomials}
So, in the same way, for any positive real number $r$, there exists an integer $n_r(X)$ such that for all $n \geq n_r(X)$
$$P^{\pi}_{X^n}(r) < P_{X^n}(r).$$
Surely the integer $n_r(X)$ depends on the choice of the number $r$. In fact, in \cite[Theorem 1.2]{LY1} we showed the existence of some integer which does not depend on the choice of the chosen number $r$, as follows:
\begin{thm}
\label{pp}
Let $\epsilon$ be a positive real number. Then there exists a positive integer $n(\epsilon)$ such that for any $n \geq n(\epsilon)$
\begin{equation}\label{p<p}
P_{X^n}^{\pi}(t) < P_{X^n}(t), \forall t \in [\epsilon, \infty).
\end{equation}
\end{thm}
\begin{rem} In the above inequality (\ref{p<p}) $[\epsilon, \infty)$ cannot be replace by $[0,\infty)$ and
it is crucial that $\epsilon$ is \emph{positive}.
\end{rem}
\begin{defn}[A. Libgober and S. Yokura \cite{LY1}]
Let $\frak{pp}(X;\epsilon)$ be the \emph{smallest one} of the integer $n(\epsilon)$ such that for any $n \geq n(\epsilon)$
$$P_{X^n}^{\pi}(t) < P_{X^n}(t), \forall t \in [\epsilon, \infty).$$
It is called the \emph{stabilization threshold}. 
\end{defn}
In particular, for $\epsilon =1$,  for example we have:
\begin{enumerate}
\item $\frak{pp}(S^{2n+1};1)=1$, 
\item $\frak{pp}(S^{2n};1)=3$, 
\item $\frak{pp}(\mathbb CP^1,1)=3$ and $\frak{pp}(\mathbb CP^n,1)=2$ if $n \ge 2$. 
\end{enumerate}

So, if $X=S^n, \mathbb {CP}^n$, $\frak{pp}(X;1) \leq 3.$ We asked ourselves ``$\frak{pp}(X;1) \leq 3$ for any $X$?" Many trials did not work. However, using Hilali conjecture, we \cite[Theorem 3.1]{LY1} could show:
\begin{thm} If $X$ is a rationally elliptic space X satisfying the Hilali conjecture, then 
\begin{equation}\label{<3}
\frak{pp}(X;1) \leq 3.
\end{equation}
\end{thm} 
Without using the Hilali conjecture, we \cite[Proposition 3.7]{LY1} have the following theorem  \emph{via the formal dimension}:
\begin{thm} If $X$ is a rationally elliptic space X of formal dimension $n \geq 3$, then 
$$\frak{pp}(X;1) \leq n.$$
\end{thm} 
Note that we need $n \geq 3$, because the formal dimension of $\mathbb {CP}^1 =2$, but $\frak{pp}(\mathbb CP^1,1)=3$.
\begin{cor}
For any rationally elliptic complex algebraic variety of complex dimension $n$, $\frak{pp}(X;1) \leq 2n.$
\end{cor}
\begin{rem}
If we could find an example such that $\frak{pp}(X;1) =4$ or $\frak{pp}(X;1) > 4$, then it would be a counterexample to the Hilali conjecture. 
\end{rem}
\begin{qu}
Is there any ``reason" for why we have $3$ in the above inequality (\ref{<3})?
\end{qu}
\section{``Hilali conjecture modulo products" for mixed Hodge polynomials}
Using mixed Hodge structures on the cohomology group and the dual of homotopy group, one can define the following (co)homological) mixed Hodge polynomial $MH_X(t,u,v)$ and the homotopical mixed Hodge polynomial $MH_X^{\pi}(t,u,v)$ of a complex algebraic variety $X$:
$$MH_X(t,u,v) :=\sum_{k,p,q} \dim \Bigl ( Gr_{F^{\bullet}}^{p} Gr^{W_{\bullet}}_{p+q} H^k (X;\mathbb C)  \Bigr) t^{k} u^p v^q,$$
$$MH^{\pi}_X(t,u,v) :=\sum_{k,p,q} \dim  \Bigl (Gr_{\tilde F^{\bullet}}^{p} Gr^{\tilde W_{\bullet}}_{p+q} ((\pi_k(X)\otimes \mathbb C)^{\vee})\Bigr ) t^ku^p v^q.$$

Here  $(W_{\bullet}, F^{\bullet})$ is the mixed Hodge structure of the cohomology group and $(\widetilde W_{\bullet}, \widetilde F^{\bullet})$ is the mixed Hodge structure of the dual of homotopy groups.
There exists a mixed Hodge structure on the homotopy groups of complex algebraic varieties as well (see \cite{Mo}, \cite{Hai1, Hai2}). 

Here we note that the Poincar\'e polynomials $P_X(t)$ and $P^{\pi}_X(t)$ are the specializations of these mixed Hodge polynomials:$P_X(t)=MH_X(t,1,1), P^{\pi}_X(t)=MH^{\pi}_X(t,1,1).$

Then we \cite[Theorem 1.10]{LY1} showed the following: 
\begin{thm}
 Let $\epsilon$ be a positive real number. Then there exists a positive integer $n(\epsilon)$ such that for any $n \geq n(\epsilon)$
$$MH_{X^n}^{\pi}(t,u,v) < MH_{X^n}(t,u,v), \forall (t,u,v) \in [\epsilon, r] \times [\epsilon,r] \times [\epsilon,r].$$
\end{thm}
\begin{rem} At the moment we do no know whether $[\epsilon, r] \times [\epsilon,r] \times [\epsilon,r]$ can be replaced by $[\epsilon, \infty) \times [\epsilon, \infty) \times [\epsilon, \infty)$.
\end{rem}
{}

In fact, we \cite[Theorem 4.2]{LY1} showed the following theorem, which is a stronger version of 
Theorem \ref{bmm} (Biswas--Mu\~noz--Murillo's theorem \cite{BMM}):
\begin{thm}\label{mhh} The homotopical and cohomological mixed Hodge polynomials of a rationally elliptic toric manifold $X$ of complex dimension $n$ is equal to those of a product of complex projective spaces, i.e.,
$$MH_X(t,u,v) = \prod_{i=1}^k MH_{\mathbb {CP}^{n_i}}(t,u,v) = \prod_{i=1}^k \left (1+t^2uv+ \cdots + t^{2i}(uv)^i + \cdots + t^{2n_i}(uv)^{n_i} \right ).$$
$$MH^{\pi}_X(t,u,v) = \sum_{i=1}^k MH^{\pi}_{\mathbb {CP}^{n_i}}(t,u,v) = \sum_{i=1}^k \left (t^2uv+ t^{2n_i+1}(uv)^{n_i+1} \right ).$$
\end{thm}

\begin{con}[by A. Libgober]\label{Lib-con} If a quasi-projective variety is rationally elliptic, then its mixed Hodge structure is of Hodge--Tate type (or called ballanced \cite{FS}), i.e., the non-zero mixed Hodge numbers are $h^{k,p,q}$ with $p=q$. 
\end{con}
If Conjecture \ref{Lib-con} is correct, then for any rationally elliptic quasi-projective variety $X$ the mixed Hodge polynomial $MH_X(t,u,v)$ is a polynomial of $t$ and $uv$, as in Theorem \ref{mhh}.\\

\noindent
{\bf Acknowledgements:} This is an extended version of the author's talk at Mini workshop ``Various aspects of singularities" held at the University of Tokyo, March 30--31, 2023. The author would like to thank the organizers, Toru Ohmoto, Hiraku Kawanoue, Kei-ichi Watanabe and Shihoko Ishii, for the invitation to the workshop and in particular Prof. S. Ishii for support for travel expenses. The author would also like to thank Anatoly Libgober and Alex Dimca for useful comments and suggestions. The author is supported by JSPS KAKENHI Grant Numbers JP19K03468 and 23K03117.


\begin{thebibliography}{ams}
\bibitem{AB}
J. Amor\'os and I. Biswas, {\it Compact K\"ahler manifolds with elliptic homotopy type}, Adv. Math., {\bf 224} (2010), 1167--1182.

\bibitem{BD}
G. Barthel and A. Dimca, {\it On conplex projective hypersurfaces which are homology-$P^n$'s}, in \emph{Singularities}, ed. by J.-P. Brasselet, London Math. Soc. Lecture Note Ser.,  {\bf 201} (1994), 1--28.

\bibitem{BB} M. Berger and R. Bott, {\it Sur kes vari\'et\'es \`a courbure strictement positive}, Topology, {\bf 1} (1962), 301--311.

\bibitem{BG}  R. L. Bishop and S. I. Goldberg, {\it Some implications on the generalized Gauss-Bonnet theorem}, Trans. Amer. Math. Soc., {\bf 112} (1964), 508-545.

\bibitem{BMM} 
I. Biswas, V. Mu\~noz and A. Murillo, {\it Rationally elliptic toric varieties}, Tokyo J. Math., {\bf 44} (1) (2021) 235--250.

\bibitem{Br} L. Brenton, {\it Some examples of singular compact analytic surfaces which are homotopy equivalent to the complex projective plane}, Topology, {\bf 16} (1977), 423-433.

\bibitem{Car} H. Cartan,  {\it La transgression dans un groupe de Lie et dans un espace fibre principal}, Colloque de Topologie (espaces fibr\'es), Bruxelles (1950), Thone, Li\`ege; Masson, Paris, 1951,57--71.

\bibitem{CatMig} M. A. A. de Catalodo and L. Migliorini, {\it The decomposition theorem, perverse sheaves and the topology of algebraic maps}, Bull. Amer. Math. Soc., {\bf 46}, No.4 (2009), 535--633.

\bibitem{CC} D. Chataur and J. Cirici, {\it Rational homotopy of complex projective varieties with normal isolated singularities}, Forum Math., {\bf 29} (2017), 41--57.

\bibitem{Chern} S.-S. Chern, {\it The geometry of G-structures}, Bull. Amer. Math. Soc., {\bf 72} (1966), 167-2019.

\bibitem{CD} A. D. R. Choudary and A. Dimca, {\it Singular complex surfaces in $\mathbb P^3$ having the same $\mathbb Z$-homology and $\Q$-homotopy type as $\mathbb P^2$}, Bull. London Math. Soc. 22 (1990) 145--147.

\bibitem{CHHZ} S. Chouingou, M. A. Hilali, M. R. Hilali and A. Zaim, {\it Note on a relative Hilali conjecture}, Ital. J. Pure Appl. Math., {\bf 46} (2021), 115--128.

\bibitem{DGMS}
P. Deligne, P. B. Griffiths, J. W. Morgan and D. P. Sullivan, {\it Real homotopy theory of K\"ahler manifolds}, Invent. Math., {\bf 29} (1975) 245--274.

\bibitem{CatMil}
S. Cattalani and A. Milivojevi\'c, {\it Verifying the Hilali conjecture up to formal dimension twenty},
J. Homotopy Relat. Struct., {\bf 15} (2020), 323--331.

\bibitem{Chen} X. Chen, {\it Riemannian manifolds with entire Grauert tube are rationally elliptic},  Geom. \& Topol., {\bf 28}, No.3 (2024), 1099--1112.

\bibitem{FH} Y. F\'elix and S. Halperin, {\it Rational homotopy theory via Sullivan models: a survey}, Notices of the ICCM (International Consortium of Chinese Mathematicians), {\bf 5}, No.2 (2017), 14--36.

\bibitem{FHT} Y. F\'elix, S. Halperin and J.-C. Thomas, {\it Rational Homotopy Theory}, Grad. Texts in Math., {\bf 205}, Springer, 2001.

\bibitem{FOT} Y. F\'elix, J. Oprea and D. Tanr\'e, {\it Algebraic Models in Geometry}, Oxf. Grad. Texts Math., {\bf 17}, Oxford, 2008.

\bibitem{FS}
C. Florentino and J. Silva, {\it Hodge--Deligne polynomials of character varieties of free abelian groups}, Open Mathematics, {\bf 19}, No.1  (2021), 338--362.

\bibitem{FH} J. B. Friedlander and S. Halperin, {\it An arithmetic characterization of the rational homotopy groups of certain spaces}, Invent. Math., {\bf 53} (1979), 117--133.


\bibitem{Gromov} M. Gromov, {\it Curvature, diameter and Betti numbers}, Comment. Math. Helv., {\bf 56}, No.2 (1981), 179--195.

\bibitem{GH} K. Grove and S. Halperin, {\it Contributions of rational homotopy theory to global problems in geometry}, Publ. Math. Inst. Hautes \'Etudes Sci., tome {\bf 56} (1982), 171--177.

\bibitem{Hai1}
R. Hain, {\it The de Rham homotopy of complex algebraic varieties, I,}  K-Theory, {\bf 1} (1987), 271--324.

\bibitem{Hai2}
R. Hain, {\it The de Rham homotopy of complex algebraic varieties, II}, K-Theory, {\bf 1} (1987), 481--497.

\bibitem{Hal} S. Halperin, {\it Finiteness in minimal models}, Trans. Amer. Math. Soc., {\bf 230} (1977), 173--199.

\bibitem{Hal2} S. Halperin, {\it Spaces whose rational homology and de Rham homotopy are both finite dimensional"}, Ast\'erisque, tome {\bf 113-114} (1984), p. 198-205.

\bibitem{Ha} A. Hatcher, {\it Algebraic Topology}, Cambridge Univ. Press (2002).

\bibitem{Hes} K. Hess, {\it A History of Rational Homotopy Theory}, Chapter 27 of ``History of Topology" ed by I.M. James, 757--796.

\bibitem{H} M. R. Hilali,
{\it Action du tore ${\mathbb T}^n$ sur les espaces simplement connexe},
Thesis, Universit\'{e} catholique de Louvain, Belgique,  (1990). 

\bibitem {HM} M. R. Hilali and M. I. Mamouni, {\it A conjectured lower bound for the cohomological dimension of elliptic spaces}, J. Homotopy Relat. Struct., {\bf 3} (2008), 379--384.

\bibitem {HM2} M. R. Hilali and M. I. Mamouni, {\it A lower bound of cohomological dimension for an elliptic space}, Topology Appl., {\bf 156} (2008), 274--283.

\bibitem{Hopf} H. Hopf, {\it Sulla geometria riemanniana globale della superficie}, 
Seminario Mat. e Fis. di Milano, {\bf 23} (1953), 48--63.

\bibitem{KP} M. Koras and K. Palka, {\it Complex planar curves homeomorphic to a line have at most four singular points}, J. Math. Pures Appl., {\bf 158 }(2022), 144-182.

\bibitem{LS} L. Lempert and R. Sz\"oke, {\it Global solutions of the homogeneous complex Monge-Amp\`ere equation and complex structures on the tangent bundle of Riemannian manifolds}, Math. Ann., {\bf 290} (1991), 689--712.

\bibitem{LY1} A. Libgober and S. Yokura, {\it Ranks of homotopy and cohomology groups for rationally elliptic spaces and algebraic varieties}, Homology, Homotopy Appl., {\bf 24} (2) (2022), 93-113.

\bibitem{LY2} A. Libgober and S. Yokura, {\it Self-products of rationally elliptic spaces
and inequalities between the ranks of homotopy and homology groups}, Topology Appl., {\bf 313} (2022),107986 (23 pages) 

\bibitem{Mo} J. W. Morgan, {\it The algebraic topology of smooth algebraic varieties}, Publ. Math. Inst. Hautes \'Etudes Sci., {\bf 48} (1978), 137--204.

\bibitem{NY}
O. Nakamura and T. Yamaguchi, {\it Lower bounds of Betti numbers of elliptic spaces with certain formal dimensions},  Kochi J. Math, {\bf 6} (2011), 9--28.

\bibitem{Pa} A. V. Pavlov, {\it Estimates for the Betti numbers of rationally elliptic spaces}, Sib. Math. J., {\bf 43}, No.6 (2002), 1080--1085.

\bibitem{Se0} J.-P. Serre, {\it Homologie singuli\`ere des espaces fibr\'es: Applications}, Ann. of Math., {\bf 54} (1951), 425--505.

\bibitem{Se} J.-P. Serre, {\it Groupes d'homotopie et classes de groupes abelien}, Ann. of Math., {\bf 58} (1953), 258--294. 

\bibitem{Severi}  F. Severi, {\it Some remarks on the topological characterization of algebraic surfaces}, in \emph{Studies in Mathematics and Mechanics: Presented to Richard Von Mises} (Academic Press, New York), (1954), 54--61.

\bibitem{SuYa} Y. Su and J. Yang,  {\it Some results on compact K\"ahler manifolds with elliptic homotopy type}, 2022, 
arXiv:2208.06998v1 [math.AT]

\bibitem{Sul} D. Sullivan, {\it Infinitesimal computations in topology}, Publ. Math. Inst. Hautes \'Etudes Sci., {\bf 47} (1977), 269--331.

\bibitem{YY1} T. Yamaguchi and S. Yokura, {\it On a Relative Hilali Conjecture}. Afr. Diaspora J. Math. (N.S.)
{\bf 21} (2018), 81--86.

\bibitem{YY2} T. Yamaguchi and S. Yokura, {\it Poincar\'e polynomials of a map and a relative Hilali conjecture}, 
Tbil. Math. J., {\bf 13} (2020), 33--47.

\bibitem{Yau} S.-T. Yau, {\it Calabi's conjecture and some new results in algebraic geometry}, Proc. Natl. Acad. Sci. USA., {\bf 74} (5): 1798--1799. 

\bibitem {Yo} S. Yokura, {\it The Hilali conjecture on product of spaces}, Tbil. Math. J., {\bf 12}, No.4 (2019), 123--129.

\bibitem {Yo2} S. Yokura, {\it Local comparisons of homological and homotopical mixed Hodge polynomials}, Proc. Japan Acad. Ser. A Math., {\bf 96} (2020), 28--31.

\bibitem {Yo3} S. Yokura, {\it A remark on relative Hilali conjectures}, Ital. J. Pure Appl. Math., {\bf 51} (2024), 560--582.

\bibitem{ZCH} A. Zaim, S. Chouingou and M. A. Hilali, {\it On a conjecture of Yamaguchi and Yokura}, Internat. J. Math. Math. Sci., Vol. {\bf 2020}, Article ID 3195926, 5 pages,  https://doi.org/10.1155/2020/3195926.

\end{thebibliography}
\end{document}